\documentclass[12pt]{amsart}

\newtheorem{theorem}{Theorem}[section]

\newtheorem{corollary}[theorem]{Corollary}

\newcommand{\bF}{\mathbb F}

\newcommand{\ex}{ex}
\newcommand{\rank}{rank}
\newcommand{\PG}{PG}
\newcommand{\AG}{AG}
\newcommand{\GF}{GF}
\newcommand{\G}{G}
\newcommand{\del}{\setminus}

\begin{document}

\sloppy

\title[The Erd\H os-Stone theorem for finite geometries]{An analogue of
the Erd\H os-Stone theorem for finite geometries}

\author[Geelen]{Jim Geelen}
\address{Department of Combinatorics and Optimization,
University of Waterloo, Waterloo, Canada}
%\email{jfgeelen@math.uwaterloo.ca}
%\thanks{ This research was partially supported by a grant from the
%Office of Naval Research [N00014-10-1-0851].}

\author[Nelson]{Peter Nelson}
\address{School of Mathematics, Statistics and Operations Research,
Victoria University, Wellington, New Zealand}

\subjclass{05B35}
\keywords{matroids, Erd\H os-Stone Theorem,
Bose-Burton Theorem, Hales-Jewett Theorem, projective geometries,
critical exponent}
\date{\today}

\begin{abstract}
For a set $G$ of points in $\PG(m-1,q)$, let $\ex_q(G;n)$
denote the maximum size of a collection of points in 
$\PG(n-1,q)$ not containing a copy of $G$,
up to projective equivalence.  We show that 
$$
\lim_{n\rightarrow \infty} \frac{ \ex_q(G;n) }{|\PG(n-1,q)|} 
= 1-q^{1-c},$$
where $c$ is the smallest integer such that there is a
rank-$(m-c)$ flat in $\PG(m-1,q)$ that is disjoint from $G$.
The result is an elementary application of the density version of 
the Hales-Jewett Theorem.
\end{abstract}

\maketitle

\section{Introduction}

Note that if $M$ is a rank-$(r-c+1)$ flat of $\PG(r-1,q)$, then 
$|M| = \frac{q^{r-c+1}-1}{q-1}$ and each rank-$m$ flat of $\PG(r-1,q)$
intersects $M$ in a flat of rank at least $m-c+1$.
Our main result is the following:
\begin{theorem}[Main Theorem]\label{main}
For each prime-power $q$, all integers $m> c\ge 0$, and
any real number $\epsilon>0$,  there is an integer
$R=R_{\ref{main}}(m,q,c,\epsilon)$ such that,
if $n>R$ and $G$ is a set of points in $\PG(n-1,q)$ with
$|G|\le  (1-\epsilon)\left(\frac{q^{n-c+1}-1}{q-1}\right)$,
then there exists a rank-$m$ flat $F$ of $\PG(n-1,q)$
such that $\rank(F\cap G)\le m-c$.
\end{theorem}

We were motivated by a problem in extremal matroid theory
posed by Kung~[\ref{kung}]; the matroidal origins
of the problem are reflected in our terminology which 
we briefly review below.

Let $\bF$ be a finite field of order $q$ and let $V$
be a rank-$r$ vector space over $\bF$.
A {\em rank-$k$ flat} of $\PG(r-1,\bF)$ is
a $(k+1)$-dimensional subspace of $V$;
the {\em points} are the rank-$1$ flats;
the {\em lines} are the rank-$2$ flats; and
the {\em hyperplanes} are the rank-$(r-1)$ flats.
Technically the projective geometry depends on the particular 
vector space $V$; to make this explicit, we write
$\PG(V)$ for the projective geometry given by $V$.

We refer to a set $H$ of points in $\PG(r-1,\bF)$, for some $r$,
as a {\em geometry over $\bF$} and we define 
$\rank(H)$ to be the rank of the flat spanned by $H$.
If $H$ and $G$ are geometries over $\bF$, then 
there are vector spaces $V_1$ and $V_2$ over $\bF$
so that $H$ is a spanning set of points in $\PG(V_1)$ and
$G$ is  a spanning set of points in $\PG(V_2)$.
We say that {\em $H$ is a restriction of $G$}
or that {\em $G$ contains $H$}, if
there is a rank-preserving projective transformation
from $V_2$ to a vector space $V'_2$ containing $V_1$
so that $H$ is contained in the image of $G$.

For a geometry $H$ over $\bF$ and positive integer $n$,
we let $\ex_q(H;n)$ denote the maximum number of points
in a rank-$n$ geometry over $\bF$ not containing $H$.

For integers $0\le c \le m$,
let $F$ be a rank-$(m-c)$ flat of $\PG(m-1,q)$ and let
$\G(m-1,q,c)$ be the geometry obtained by restricting 
$\PG(m-1,q)$ to the complement of $F$; thus
$\G(m-1,q,m) = \PG(m-1,q)$ and $\G(m-1,q,1) = \AG(m-1,q)$, the rank-$m$ affine geometry over $GF(q)$.
The {\em critical exponent} of $H$ over $\GF(q)$, written $c(H;q)$, is
the minimum $c$ such that $H$ is contained in $\G(r(M)-1,q,c)$. 
%If $H$ is contained in
%$\G(r(M)-1,q,c)$ but not in $\G(r(M)-1,q,c-1)$,
%then we say that $H$ has {\em critical exponent} $c$ and let
%$c(H;q)=c$.  
The critical exponent was introduced by
Crapo and Rota~[\ref{cr}] and is related to the chromatic number 
of a graph. 

The following result, which is an easy corollary of 
Theorem~\ref{main}, was all but conjectured by Kung~[\ref{kung}].
\begin{theorem}\label{cor}
Let $\bF$ be a finite field of order $q$. If $H$ is a geometry
 over $\bF$ with with critical exponent $c>0$, then
$$ \lim_{n\rightarrow \infty} \frac{ \ex_q(H;n) }{|\PG(n-1,q)|} 
= 1-q^{1-c}.$$
\end{theorem}

This theorem bears a striking resemblance to the following theorem
of Erd\H os and Stone~[\ref{es}].
For a graph $H$, let $\ex(H;n)$ denote the 
maximum number of edges in a simple $n$-vertex
graph that does not contain a subgraph isomorphic to $H$.
The {\em chromatic-number}, $\chi(G)$, of a graph $G$ is the
minimum number of colours needed to colour the vertices so that
no two adjacent vertices get the same colour.
\begin{theorem}[Erd\H os-Stone Theorem]\label{esthm}
For any graph $H$ with chromatic-number $\chi\ge 2$, 
$$ \lim_{n\rightarrow \infty} \frac{\ex(H;n)}{{n\choose 2}} =
1-\frac{1}{\chi-1}.$$
\end{theorem}

\section{old results}

In this section we briefly review related results.
Note that $\G(n-1,q,m-1)$ does not contain $\PG(m-1,q)$;
Bose and Burton~[\ref{bb}] showed that $\G(n-1,q,m-1)$
is extremal among geometries not containing $\PG(m-1,q)$. 
\begin{theorem}\label{bbthm}
Let $\bF$ be a field of order $q$ and $m$ and $r$ be integers 
with $n\ge m\ge 0$. Then
$$ \ex_q(\PG(m-1,\bF);\, n) = |\G(n-1,q,m-1)|.$$
\end{theorem}
Bonin and Qin~[\ref{bq}] determine $\ex_q(H;n)$ exactly for
several other interesting families of geometries.

Our main result is an easy application of the
following deep result due to Furstenberg and Katznelson~[\ref{fur}, Theorem 9.10] in 1985.
\begin{theorem}\label{mdhjthm}
For each field $\bF$ of order $q$, integer $ m\ge 2$, and
real number $\epsilon > 0$,  there is an integer
$R=R_{\ref{mdhjthm}}(m,q,\epsilon)$ such that,
$$ \ex_q(\AG(m-1,\bF);\, n) < \epsilon |\PG(n-1,q)|$$
for all $n> R$.
\end{theorem}
This result can be obtained as an easy application of the
density version of the multidimensional Hales-Jewett theorem, also proved by Furstenberg and Katznelson~[\ref{fk}], in 1991,
using ergodic theory.  An easier proof was later obtained via the
polymath project~[\ref{polymath}].
The ``easier proof'' is still, however,  more than 30 pages long.
Bonin and Qin~[\ref{bq}] have a much simpler proof 
of Theorem~\ref{mdhjthm} in the case that $q=2$.

\section{New results}

We start with a proof of Theorem~\ref{main}; for convenience we
restate it in a complementary form.
(The equivalence between the two statements is easy and
is left to the reader.)
\begin{theorem}[Reformulation of Theorem~\ref{main}]\label{main2}
For any integers $m> c\ge 1$ and
real number $\epsilon>0$,  there is an integer
$R=R_{\ref{main2}}(m,q,c,\epsilon)$ such that,
$$\ex_q( \G(m-1,q,c); n) <  (1-q^{1-c}+\epsilon)|\PG(n-1,q)|,$$
for all $n>R$.
\end{theorem}

\begin{proof}
Let $m> c\ge 1$ be integers and let 
$\epsilon>0$ be a real number.
The proof is by induction on $c$;
the case that $c=1$ follows directly from Theorem~\ref{mdhjthm}.
Assume that $c>1$ and that the result holds for $c-1$.

Let $r = R_{\ref{mdhjthm}}(m-c+1,q,\epsilon/2)$,
let $t$ be sufficiently large so that 
$q^{1-c}(q^r-1) \le \tfrac{\epsilon}{2}(q^n-q^r)$ for all $n > t$,
and define
$$R_{\ref{main2}}(m,q,c,\epsilon)=
\max(t,R_{\ref{main2}}(r,q,c-1,q^{2-c}-q^{1-c})).$$
Now let $n > R_{\ref{main2}}(m,q,c,\epsilon)$ and
let $M$ be a restriction of $\PG(n-1,q)$ with
$|M|\ge  (1-q^{1-c}+\epsilon)|\PG(n-1,q)|$.

By the inductive assumption, $M$ has a $\G(r-1,q,c-1)$-restriction.
Thus there are flats $F_0\subseteq F_1$ of $\PG(n-1,q)$ such that
$\rank(F_1)=r$, $\rank(F_0)=r-c+1$, and $F_1-F_0\subseteq M$.  Let
$F_0^c\subseteq F_1$ be a rank-$(c-1)$ flat that is disjoint from $F_0$.

Note that, by our definition of $t$,
$$|M\del F_1|\ge \left(1-q^{1-c} +\tfrac{\epsilon}{2}\right)
|\PG(n-1,q)-F_1|.$$
So by an elementary averaging argument, 
there exists a rank-$(r+1)$ flat $F_2$ containing 
$F_1$ such that 
$$|M\cap(F_2-F_1)| \ge   \left(1-q^{1-c} +\tfrac{\epsilon}{2}\right)|F_2-F_1| = 
(1-q^{1-c})q^{r} +\tfrac{\epsilon}{2} q^r .$$ 

We want to find a rank-$m$ flat $F\subseteq F_2$
such that $F^c_0\subseteq F\not\subseteq F_1$ and $F-F_1\subseteq M$.
If $F$ satisfies these conditions, then
$\rank(F \cap F_0) = m-c$ and, hence,
the restriction of  $M$ to $F$  contains $\G(m-1,q,c)$.

Let $S= (F_2-F_1)\cap M$.
For a flat $F$ of $\PG(n-1,q)$ and point $e\not\in F$,
we let $F+e$ denote the flat spanned by $F\cup\{e\}$.
Let $e\in F_2-F_1$ and let $Q = (F_0+e)-F_0$.
Now, for each $f\in Q$,
let $S_f = (F_0^c + f)\cap S$.
Note that $(S_f\, : \, f\in Q)$ partitions $S$ and
$|S_f|\le q^{c-1}$.  Finally, let $Q_1$ be the set of all $f\in Q$,
such that $|S_f|= q^{c-1}$.

All  vectors in $Q-Q_1$
extend to at most $q^{c-1}-1$ elements in $S$, so
\begin{eqnarray*}
  (q^{c-1}-1) q^{r-c+1} + |Q_1| &\ge& |S| \\
 &\ge & \left( 1 -q^{1-c}+\tfrac{\epsilon}{2} \right) q^r\\
 &= & ( q^{c-1}-1) q^{r-c+1}+\tfrac{\epsilon}{2}  q^r.
\end{eqnarray*}
Thus $|Q_1|\ge \frac{\epsilon}{2} q^r$.
By Theorem~\ref{mdhjthm},
there is a subset $Q_2$ of $Q_1$ such that 
$Q_2 \cong \AG(m-c, q)$.
Let $F$ be the flat of $\PG(n-1,q)$ 
spanned by $F_0^c$ and $Q_2$.
Thus $F$ has rank $m$, $F_0^c\subseteq F$, and, since $Q_2\subseteq Q_1$,
$F-F_1\subseteq M$.
So the restriction of $M$ to $F-F_0$ gives $\G(m-1,q,c)$.
\end{proof}

We can now prove Theorem~\ref{cor}, which we restate here 
for convenience.
\begin{corollary}
Let $\bF$ be a finite field of order $q$. If $H$ is a geometry 
over $\bF$ with with critical exponent $c>0$, then
$$ \lim_{n\rightarrow \infty} \frac{ \ex_q(H;n) }{|\PG(n-1,q)|} 
= 1-q^{1-c}.$$
\end{corollary}

\begin{proof}
Observe that $H$ is a restriction of 
$\G(r(N)-1,q,c)$ but it is not a restriction of
$\G(n-1,q,c-1)$.
Then, by Theorem~\ref{main2}, for
all $\epsilon>0$ and all
sufficiently large $n$,
$$
\frac{q^n}{q^n-1}(1-q^{1-c})=
\frac{|\G(n-1,q,c-1)|}{|\PG(n-1,q)|}\le
 \frac{\ex_q(H;n)}{|\PG(n-1,q)|} \le 1-q^{1-c} +\epsilon,$$
so the result holds.
\end{proof}

The value of $R_{\ref{main2}}(m,q,c,\epsilon)$ provided by Theorem~\ref{main2}
depends on that of $R_{\ref{mdhjthm}}(m,q,\epsilon)$, 
for which the bounds in [\ref{polymath}] are Ackermann-like 
 for all $q > 2$. 
In the binary case, however, the main theorem 
of [\ref{bq}] implies that the relatively small function 
$R_{\ref{mdhjthm}}(m,2,\epsilon) = 2^{m-2}\lceil 1-\log_2 \epsilon\rceil$ 
will satisfy Theorem~\ref{mdhjthm}. From this, 
one can derive from the proof that 
$R_{\ref{main2}}(m,2,c,\epsilon) = T_c(m+d)$ will satisfy Theorem~\ref{main2}, where 
$d = \lceil\log_2\lceil(2-\log_2 \epsilon)\rceil\rceil$, 
and $T_c$ is the tower function recursively defined by $T_0(s) = s$ 
and $T_i(s) = T_{i-1}(2^s)$ for all $i > 0$.

\section*{References}

\newcounter{refs}

\begin{list}{[\arabic{refs}]}%
{\usecounter{refs}\setlength{\leftmargin}{10mm}\setlength{\itemsep}{0mm}}

\item \label{bq}
J.E. Bonin, H. Qin,
Size functions of subgeometry-closed classes of
representable combinatorial geometries,
Discrete Math. 224, (2000) 37-60.

\item \label{bb}
R.C. Bose, R.C. Burton,
A characterization of flat spaces in a finite
geometry and the uniqueness of the Hamming
and MacDonald codes,
J. Combin. Theory 1, (1966) 96-104.

\item\label{cr}
H.H. Crapo, G.-C. Rota,
On the foundations of combinatorial theory:
Combinatorial geometries, M.I.T. Press, Cambridge, Mass., 1970.

\item \label{es}
P. Erd\H os, A.H. Stone,
On the structure of linear graphs,
Bull. Amer. Math. Soc. 52, (1946) 1087-1091.

\item \label{fur}
H. Furstenberg, Y. Katznelson,
An ergodic Szemeredi theorem for IP-systems and combinatorial theory,
Journal d'Analyse Math\' ematique 45, (1985) 117-168.

\item \label{fk}
H. Furstenberg, Y. Katznelson,
A density version of the Hales-Jewett Theorem,
J. d'Analyse Math\' ematique 57 (1991) 64-119.

% \item \label{ramseybook}
% R. Graham, B. Rothschild, J.H. Spencer,
% Ramsey Theory,
% John Wiley and Sons, NY, (1980).

\item\label{kung}
J.P.S. Kung,
Extremal matroid theory, in:
Graph Structure Theory, N. Robertson and P.D. Seymour, eds.
(Amer. Math. Soc., Providence RI, 1993), 21-61.

\item\label{polymath}
D.H.J. Polymath,
A new proof of the density Hales-Jewett theorem,
arXiv:0910.3926v2 [math.CO], (2010) 1-34.

% \item \label{oxley}
% J. G. Oxley,  {\em Matroid Theory},
% Oxford University Press, New York, 1992.

\end{list}

\end{document}